%% file: paper020911.tex
\documentclass[fleqn,11pt]{article}
\usepackage{amsmath,graphicx,here,latexsym,amssymb,subfigure,a4wide,chicago,bbm,url}
\title{A maximum smoothed likelihood estimator in the current status continuous mark model}
\author{Piet Groeneboom, Geurt Jongbloed \& Birgit Witte}

\input commands.tex
\input theorems.tex

\numberwithin{equation}{section}
\numberwithin{figure}{section}

\begin{document}

\maketitle
\vspace{30mm}
\begin{tabular}{lcl}
Address: & & \\[1ex]
Piet Groeneboom \& Geurt Jongbloed & & Birgit I. Witte (Corresponding author)\\
Delft University of Technology & & VU University Medical Center\\
Delft Institute of Applied Mathematics & & Department of Epidemiology and Biostatistics\\
Mekelweg 4 & & PO Box 7057\\
2628 CD Delft & & 1007 MB Amsterdam\\
The Netherlands & & The Netherlands\\
 & & {\tt B.Witte@vumc.nl}
\end{tabular}

\newpage
\begin{abstract}
We consider the problem of estimating the joint distribution function of the event time and a continuous mark variable based on censored data. More specifically, the event time is subject to current status censoring and the continuous mark is only observed in case inspection takes place after the event time. The nonparametric maximum likelihood estimator (MLE) in this model is known to be inconsistent. We propose and study an alternative likelihood based estimator, maximizing a smoothed log-likelihood, hence called a maximum smoothed likelihood estimator (MSLE). This estimator is shown to be well defined and consistent, and a simple algorithm is described that can be used to compute it. The MSLE is compared with other estimators in a small simulation study.
\end{abstract}

\vspace{5mm}
\noindent{\bf Keywords}: bivariate distribution function, censored data, pointwise consistency, histogram estimator, Kullback-Leibler divergence
\vspace{5mm}

\noindent{\bf AMS Subject Classification}: 62G05, 62G20, 62N02, 62H12.
\vspace{5mm}

\section{Introduction}
In survival analysis one is interested in the distribution of the time $X$ it takes before a certain event (failure, onset of a disease) takes place. Typically,  the variable $X$ is not observed completely, due to some sort of censoring. Depending on the censoring mechanism and the precise assumptions imposed on the distribution function $F_0$ of $X$, many estimators have been defined and studied in the literature.

In the context of case I interval censoring, \citeN{groeneboom:92} study the (nonparametric) maximum likelihood estimator (MLE). It maximizes the likelihood of the observed data over all distribution functions, without any additional constraints. In case $X$ is subject to right-censoring, the MLE is the Kaplan-Meier estimator (\citeN{kaplan:58}). In these models, the resulting estimators are piecewise constant between jumps, and therefore fail to have a density w.r.t.\ Lebesgue measure.

If the quantity of interest is bivariate, $(X,Y)$, with joint distribution function $F_0$, the situation is more complicated. If both components $X$ and $Y$ of the pair $(X,Y)$ are subject to right censoring, the MLE is inconsistent, see \shortciteN{tsai:86}, and modifications of the MLE to ensure consistency have been discussed by several authors, see, e.g., \citeN{dabrowska:88} and \citeN{laan:96aos}. In case both $X$ and $Y$ are subject to interval censoring, the MLE is consistent, see, e.g., \citeN{song:01}. In computing the MLE, one first has to determine the set of points where the MLE can have mass (which is different from the set of observations). \citeN{maathuis:05} provides two reduction algorithms that can be used to determine this set. The reason that the MLE is consistent for the bivariate current status model, and inconsistent for the bivariate right censoring model (where one has in fact more information), is that in the latter case the MLE only uses the information on ``lines", if the observation is uncensored in one coordinate, and does not use the surrounding information for the uncensored coordinate.
One would need information on the conditional distribution on these lines to distribute mass in such a way that a consistent estimate would result, but this conditional distribution is not available, since it is part of the estimation problem. In the current status model or the interval censoring model with more observation times for the ``hidden variable", one only has information on the interval to which the hidden variable belongs, and the MLE therefore automatically uses the surrounding information. For this reason a reduction of the bivariate right-censoring model to the interval censoring model has been proposed to obtain consistent estimators of the bivariate distribution function: in this way the information of a whole set of lines is combined.

An interesting situation arises when $X$ is a survival time and $Y$ a (continuous) mark variable. In case $X$ is subject to right-censoring and $Y$ is only observed if $X$ is observed, \citeN{huang:98} study a nonparametric estimator of the bivariate distribution function $F_0$. This estimator is uniformly strongly consistent and asymptotically normally distributed. \shortciteN{hudgens:07} study several estimators for the joint distribution of a survival time and a continuous mark variable, when the survival time is interval censored and the mark variable is possibly missing. In this paper a computational algorithm for the MLE is proposed, but since the MLE is inconsistent (\citeN{maathuis:08}), one would be inclined to recommend not to use this estimator.

The model we focus on in this paper, the current status continuous mark (CSCM) model, is a special case of the model studied in the latter paper, since the $X$ is subject to current status censoring, the simplest case of interval censoring, and $Y$ is only observed if the event time was before the censoring time. More precisely, instead of observing $(X,Y)$, we observe a variable $T$, independent of $(X,Y)$, as well as the variable $\Delta=1_{\{X\leq T\}}$. In case the variable $X$ is smaller than or equal to $T$, i.e.\ $\Delta=1$, we also observe the variable $Y$, in case $\Delta=0$ we do not. We denote the (bivariate) distribution function of $(X,Y)$ by $F_0$ and assume it to have a density $f_0$ w.r.t.\ Lebesgue measure. Because  $P(Y=0)=0$ under $F_0$, we can represent the observable information on $(X,Y)$ in the vector $W = (T,\Delta\cdot Y)$.

An application where observations can be modeled by this model is the HIV vaccine trial studied by \shortciteN{hudgens:07}. In these HIV vaccine trials, participants are injected with a vaccine and tested for infection with HIV during several follow-ups. Efficacy of the vaccine might depend on the genetic sequence of the exposing virus, and the so-called viral distance $Y$ between the DNA of the infecting virus and the virus in the vaccine could be considered as a continuous mark variable. In general, the time $X$ to HIV infection is subject to interval censoring case $k$, with current status censoring (or interval censoring case 1) as a special instance.

The MLE in the CSCM model is inconsistent and \citeN{maathuis:08} obtain a consistent estimator by discretizing the mark variable to $K$ levels. The resulting observations can then be viewed as observations from the current status $K$-competing risk model. Apart from consistency, global and local asymptotic distribution properties for the MLE in the latter model are proved in \shortciteANP{groeneboom:08} \citeyear{groeneboom:08,groeneboom:08b}. Asymptotic results for $K\to\infty$ as $n\to\infty$ are not yet known. Another approach to obtain a consistent estimator is adopted in \shortciteN{witte:11spie}. There a plug-in inverse estimator for the bivariate distribution function $F_0$, using kernel estimators, is studied and its asymptotic distribution is derived. In contrast to the proposed approach of \citeN{maathuis:08}, this estimator does have a Lebesgue density on $[0,\infty)^2$. Unfortunately, for finite sample size $n$, this estimator does not necessarily satisfy the conditions of a bivariate distribution function (i.e.\ each rectangle has nonnegative mass). If estimators are to be used in bootstrap experiments, this is a serious drawback, since it is not clear how to interpret sampling from such a ``distribution".

In this paper we consider an alternative method, the method of maximum smoothed likelihood. This is a natural approach since also in other models where MLE's are inconsistent (\citeN{Jongbloed09}) or nonsmooth (\shortciteN{GroJoWi10}), maximum smoothed likelihood estimators (MSLEs) provide consistent and smooth estimators. The basic idea is to replace the empirical distribution function in the log-likelihood  by a smooth estimator. We prove that for a histogram-type smoothing of the observation distribution the resulting MSLE is consistent under certain conditions. Contrary to the plug-in inverse estimator studied by \shortciteN{witte:11spie}, the MSLE is a real distribution function.

The outline of this paper is as follows. In section \ref{sec:model} we introduce the CSCM model in more detail and define the MSLEs \FMS\ and \fMS\ for the distribution function $F_0$ and its density.  In section \ref{sec:cons}, consistency of the bivariate estimator \FMS\ and the marginal estimator \FMSX\ for the distribution function of $X$ are proved. A comparative simulation study is presented in section \ref{sec:discussion}.  Technical proofs and lemmas are given in appendix \ref{app:technlem}, and in Appendix \ref{sec:EM} we also give a desription of an easy to implement EM algorithm for computing the MSLE.

\section{Model description and definition of the estimator}\label{sec:model}
In this section we describe the current status continuous mark model in more detail and define the maximum smoothed likelihood estimator (MSLE) \fMS\ for the bivariate density $f_0$. The smoothed log-likelihood is obtained by replacing the empirical distribution function in the definition of the log-likelihood by a smooth estimator $\hat H_n$ for the distribution function. We prove that for piecewise constant density estimates $\hn$ the estimator \fMS\ exists and is unique under certain conditions. Based on \fMS\ we also define the MSLE for the bivariate distribution function $F_0$ and for the marginal distribution function $F_{0,X}$ of $X$.

Consider an i.i.d.\ sequence $(X_1,Y_1), (X_2,Y_2), \ldots$ with bivariate distribution function $F_0$ on $\W = [0,\infty)\times[0,\infty)$ and independent of this an i.i.d.\ sequence $T_1, T_2, \ldots$ with distribution function $G$ and Lebesgue density $g$ on $[0,\infty)$. Based on these sequences, define $\Delta_i = 1_{\{X_i\leq T_i\}}$ and $W_i = (T_i,\Delta_i\cdot Y_i) =: (T_i,Z_i)$, where we assume that $P(Y_i=0)=0$. In words: if the event already occurred before time $T_i$, the mark variable $Y_i$ is observed; if not, $Y_i$ is not observed. Note that $\Delta_i=1_{\{Z_i>0\}}$.

Let $F_{0,X}(t) = \int_0^t\int_0^{\infty} f_0(u,v)\,dv\,du$ be the marginal distribution function of $X$ and define $\partial_2 F_0(t,z) = \frac{\partial}{\partial z}F_0(t,z) = \int_0^t f_0(u,z)\,du$. Then $W_1, W_2, \ldots$ are i.i.d.\ and have density
\be \hf(t,z) = 1_{\{z>0\}}(z) g(t)\partial_2 F_0(t,z) + 1_{\{z=0\}}(z)g(t)\big(1-F_{0,X}(t)\big) =: 1_{\{z>0\}}(z) h_1(t,z) + 1_{\{z=0\}}(z)h_0(t)\nonumber,\ee
with respect to the  measure $\lambda$ on $[0,\infty)^2$ defined below. Let $\lambda_i$ be Lebesgue-measure on $\R^i$, $\B$ the Borel $\sigma$-algebra on $[0,\infty)^2$, then the measure $\lambda$ is defined by
\bef\label{def:lambda}
\lambda\big(B\big) = \lambda_2\big(B\big) + \lambda_1\big(\{x\in[0,\infty): (x,0)\in B\}\big), \ B\in\B.
\eef

Let $\Hn{n}$ be the empirical distribution function of $W_1, \ldots, W_n$. \shortciteN{hudgens:07} define and characterize the nonparametric maximum likelihood estimator (MLE), which maximizes
\begin{align}
\label{ll:cm}
l(f) &= \int \log h_f(t,z)\,d\Hn{n}(t,z)\nonumber\\
&= \int 1_{\{z>0\}}(z)\log\big(\partial_2 F(t,z)\big) + 1_{\{z=0\}}(z)\log\big(1-F_X(t)\big)\,d\Hn{n}(t,z)
\end{align}
over the class of distribution functions with density $f$ w.r.t.\ $\lambda_1$ $\times$ counting measure on the observed marks, with appropriate interpretation of the partial derivative. However, \citeN{maathuis:08} prove that this MLE is inconsistent. The heart of the consistency difficulties with the MLE resides in the first part of the log-likelihood
$\int 1_{\{z>0\}}(z)\log\partial_2 F(t,z)\,d\Hn{n}(t,z)$,
where one really has to deal with a density type expression in $z$ (i.e., $\partial_2 F(t,z)=\int_0^t f(s,z)\,ds$) instead of a bivariate distribution function.

We now propose an alternative likelihood-based method, the method of maximum smoothed log-likelihood introduced in \citeN{eggermont:01}, where the resulting estimator for $F_0$ will turn out to be consistent. Let $\hat H_n$ be a smoothed version of the empirical distribution function $\Hn{n}$, then the smoothed log-likelihood $l^S$ is defined by replacing $\Hn{n}$ in (\ref{ll:cm}) by its smoothed version $\hat H_n$, i.e.
\be
l^S(f) = \int 1_{\{z>0\}}(z)\log\big(\partial_2 F_0(t,z)\big) + 1_{\{z=0\}}(z)\log\big(1-F_{0,X}(t)\big),d\hat H_n(t,z),
\ee
Note that the factorization property of the MLE in the current status model, by which the part involving $g$ drops out, also holds in the present case: we do not have to maximize over the unknown $g$, because it does not play a role in the maximization problem. The maximum smoothed likelihood estimator (MSLE) \fMS\ for the density $f_0$ is then defined as
\bef\label{def:fMS_cscm}
\fMS = \argmax_{f\in\F}l^S(f),
\eef
where \F\ is the class of all distribution functions with density $f$ w.r.t.\ Lebesgue measure on $[0,\infty)^2$. The MSLE for the bivariate distribution function $F_0$ is naturally defined as
\be\FMS(t,z) = \int_0^t \int_0^z \fMS(u,v)\,dv\,du,\ee
and the estimators for $F_{0,X}$ and $\partial_2 F_0$ are defined similarly,
\be \FMSX(t) = \int_0^t\int_0^{\infty} \fMS(u,v)\,dv\,du ,\qquad\partial_2 \FMS(t,z) = \frac{\partial}{\partial z} \FMS(t,z) = \int_0^t \fMS(u,z)\,du.\ee

Note that the MSLE \fMS\ can also be seen as a Kullback-Leibler projection, minimizing
\be \K\big(\hn,h_f\big) = \int \hn(t,z)\log\frac{\hn(t,z)}{h_f(t,z)}\,d\lambda(t,z) = \int \log\hn(t,z)\,d\hat H_n(t,z) - \int \log h_f(t,z)\,d\hat H_n(t,z),\ee
over densities $f\in\F$. This follows from the fact that the first term on the right-hand side does not depend on $f$ and the second term on the right-hand side equals $-l^S(f)$.

In this paper, for the ease of computations and proving consistency, we take a histogram-type estimator for the density $\hn$ of $\hat H_n$, resulting in a piecewise linear estimator $\hat H_n$. There are other possibilities as well to choose $\hat H_n$, see section \ref{sec:discussion}.  To define our estimator $\hn$ we take two binwidths $\dn$ and $\en$ and define $A_{n,i}=((i-1)\dn,i\dn]=(a_{n,i-1},a_{n,i}]$ and $B_{n,j}=((j-1)\en,j\en]=(b_{n,j-1},b_{n,j}]$ for $i=1,\ldots, k_n$, $j=1,\ldots, l_n$. Then slightly abusing notation, the estimator $\hn$ is defined and denoted by
\be\lefteqn{\hn(t,0) = \hi = \dn^{-1}\Hn{n}\big(A_{n,i}\times\{0\}\big), \ \text{if}\ t\in A_{n,i}}\\
\lefteqn{\hn(t,z) = \hij = \dn^{-1}\en^{-1}\Hn{n}\big(A_{n,i}\times B_{n,j}\big), \ \text{if}\ (t,z)\in A_{n,i}\times B_{n,j}.}\ee

We consider the estimator \fMS\ that is obtained by maximizing $l^S(f)$ over the class $\F_n$ of piecewise constant densities with cells $A_{n,i}\times B_{n,j}$. For the resulting histogram-type estimator \fMS\ we can prove that it is well defined and unique if all cells contain at least one observation. This holds with high probability under certain conditions on the observation density \hf\ and the total number $k_n\cdot l_n$ of cells $A_{n,i}\times B_{n,j}$.

Before proving the existence and uniqueness of \fMS, stated in Theorem \ref{th:exis_uniq_cm} below, we introduce some notation to relate the class of densities we consider to appropriate subsets of Euclidean space.
\bef
\lefteqn{\B_n = \big\{f\in \R^{k_n\times l_n}\ :\ 0\leq f_{i,j} \leq (\dn\en)^{-1} \ \forall i,\ j \big\}}\\
\lefteqn{\ai(f) = \dn\en \sum_{l=i}^{k_n}\sum_{j=1}^{l_n}f_{l,j},\qquad \bij(f) = \dn\sum_{l=1}^{i}f_{l,j}, \ \text{for}\ f\in\B_n}\label{def:alpha_beta}
\eef
with $0\log0 := 0$, $\alpha_{k_n+1}(f) = 0$ and $\beta_{0,j}(f)=0$ for all $j$.

\begin{lemma}\label{lem:rewr_lS_cm}
Define
\begin{align}
\psi_{-1}(f) =&\dn\sum_{i=1}^{k_n} \hi\ \phi\big(\alpha_{i+1}(f),\ai(f)\big)+ \dn\en\sum_{i=1}^{k_n}\sum_{j=1}^{l_n} \hij\ \phi\big(\bij(f),\beta_{i-1,j}(f)\big)\nonumber\\
&- \dn\en\sum_{i=1}^{k_n}\sum_{j=1}^{l_n}f_{i,j} + 1,\nonumber\\
\phi(x,y) =& \left\{\begin{array}{ll}
(x\log x-y\log y)/(x-y), & x,y\in[0,1],\ x\not=y\\
1+\log x, & x,y\in[0,1],\ x=y,\end{array}\right.\label{def:phixy}
\end{align}
then
\be \argmax_{f\in\F_n}l^S(f) = \argmax_{f\in\B_n}\psi_{-1}(f).\ee
\end{lemma}
The proof is given in the Appendix.

Using this lemma, we can prove the existence and uniqueness of \fMS\ theorem below.

\begin{theorem}\label{th:exis_uniq_cm}
If $\hat h_i>0$ and $\hat h_{i,j}>0$ for all $i,j$, then the maximizer \fMS\ of $l^S$ over $\F_n$ exists and is unique.
\end{theorem}

\begin{proof}
The function $\phi$ is continuous on $(0,1]^2$ so $\psi_{-1}$ is continuous on the compact set $\B_n$. 
Hence $\psi_{-1}$ attains its maximum over $\B_n$ and \fMS\ exists. Uniqueness of \fMS\ follows from the strict concavity of $l^S(f)$ on its domain.
\end{proof}

\begin{remark}\label{rm:pos_hnh}\rm
If $\hat h_i=0$ for some $i$ or $\hat h_{i,j}=0$ for some $i$ and $j$, we can construct examples where \fMS\ is not unique. However, we can prove under various conditions on \hf, $k_n$ and $l_n$ that
\bef\label{P:pos_hnh}
P\big(\exists\ i\ :\ \hat h_i =0\big) \conv 0, \qquad P\big(\exists\ i,j\ :\ \hat h_{i,j} =0\big) \conv 0,
\eef
so that with probability converging to one \fMS\ is well defined for $n$ sufficiently large. For example, if $k_n=l_n$ and for some $c>0$
\bef\label{cond:h}
\int_{A_{n,i}} h_0(t)\,dt \geq \frac{c}{k_n^2}, \qquad \int_{A_{n,i}\times B_{n,j}} h_1(t,z)\,dt\,dz \geq \frac{c}{k_n^3},
\eef
for all $i$ and $j$, then
\begin{align*} P\big(\exists\ i,j\ :\ \hat h_{i,j} =0\big) &\leq \sum_{i,j}P\big(\hat h_{i,j} = 0\big) \leq k_n^2 \left(1-\frac{c}{k_n^3}\right)^n = k_n^2 \left[\left(1-\frac{c}{k_n^3}\right)^{k_n^3}\right]^{n/k_n^3}\\
&\approx k_n^2e^{-cn/k_n^3} \conv 0,\end{align*}
if $n^{-1}k_n^3\log k_n\to 0$.\\
A similar argument shows that also $P\big(\exists\ i\ :\ \hat h_i =0\big) \leq k_n \left(1-\frac{c}{k_n^2}\right)^n \conv 0$.\\
Assume for example that $f_0$ has compact support $\W_M = [0,M_1]\times[0,M_2]$ for some constants $0<M_1,M_2<\infty$ and stays away from zero on its support, and that $g\geq\kappa_1>0$ on $[0,M_1]$. Then (\ref{cond:h}) is satisfied. This condition is far from necessary, but only meant as an illustration for a condition under which (\ref{cond:h}) is satisfied.
\end{remark}

By Theorem \ref{th:exis_uniq_cm} we know conditions under which the estimator \fMS\ defined in (\ref{def:fMS_cscm}) exists and is unique. A simple EM algorithm for computing the MSLE is given in the appendix. 

\begin{figure}[hp]
\centering
\subfigure{
 \includegraphics[width=0.45\textwidth]{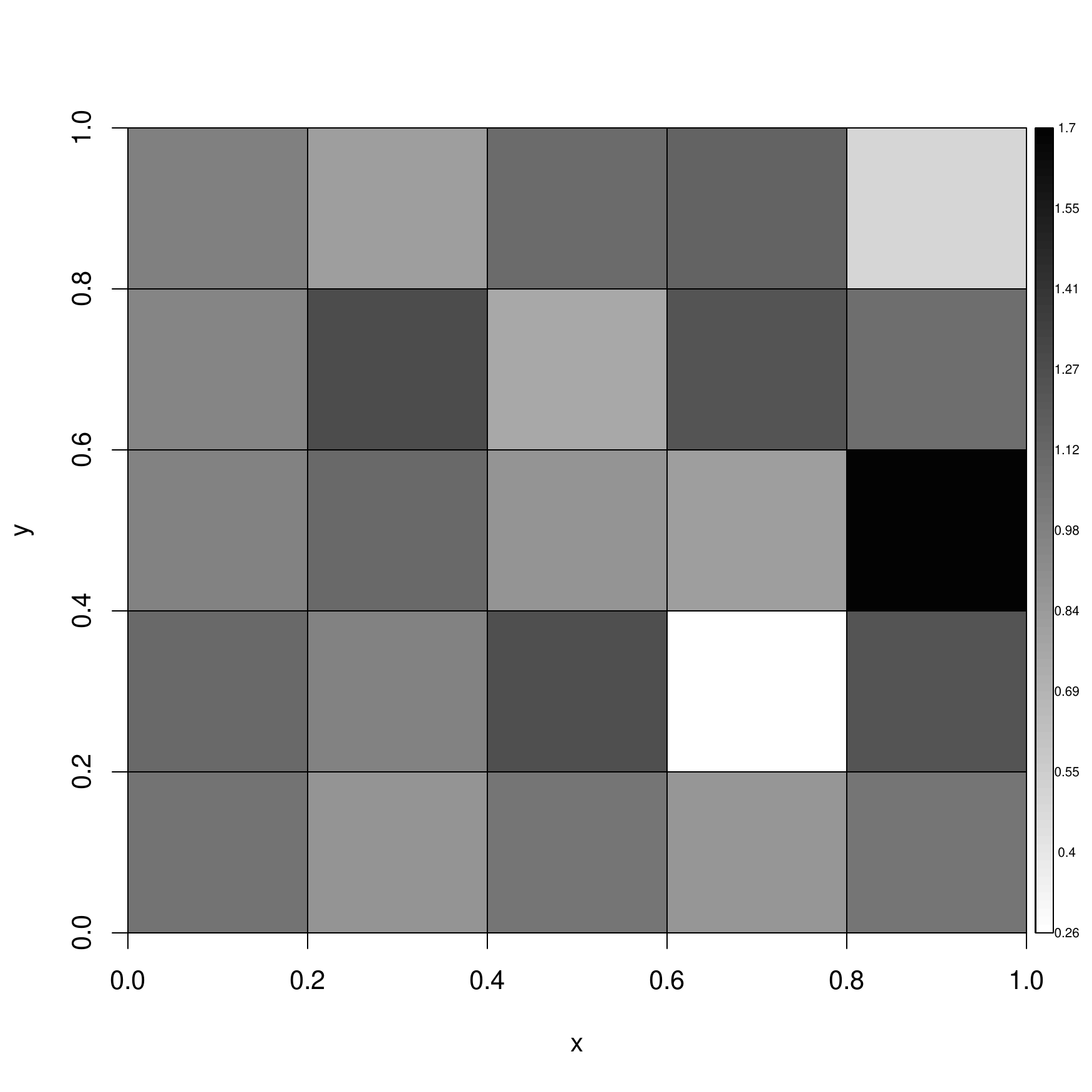} }
\subfigure{
 \includegraphics[width=0.45\textwidth]{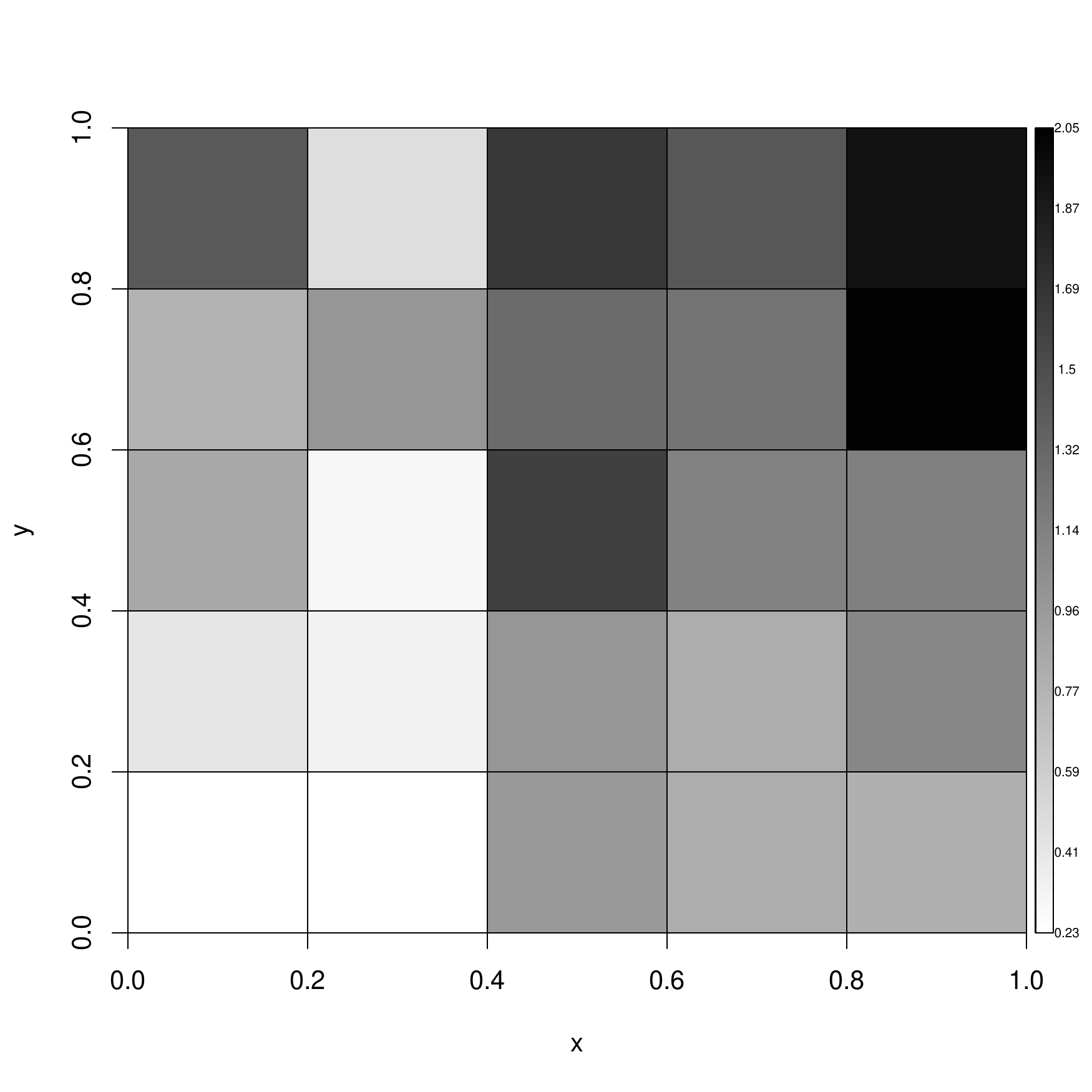} }
\subfigure{
 \includegraphics[width=0.45\textwidth]{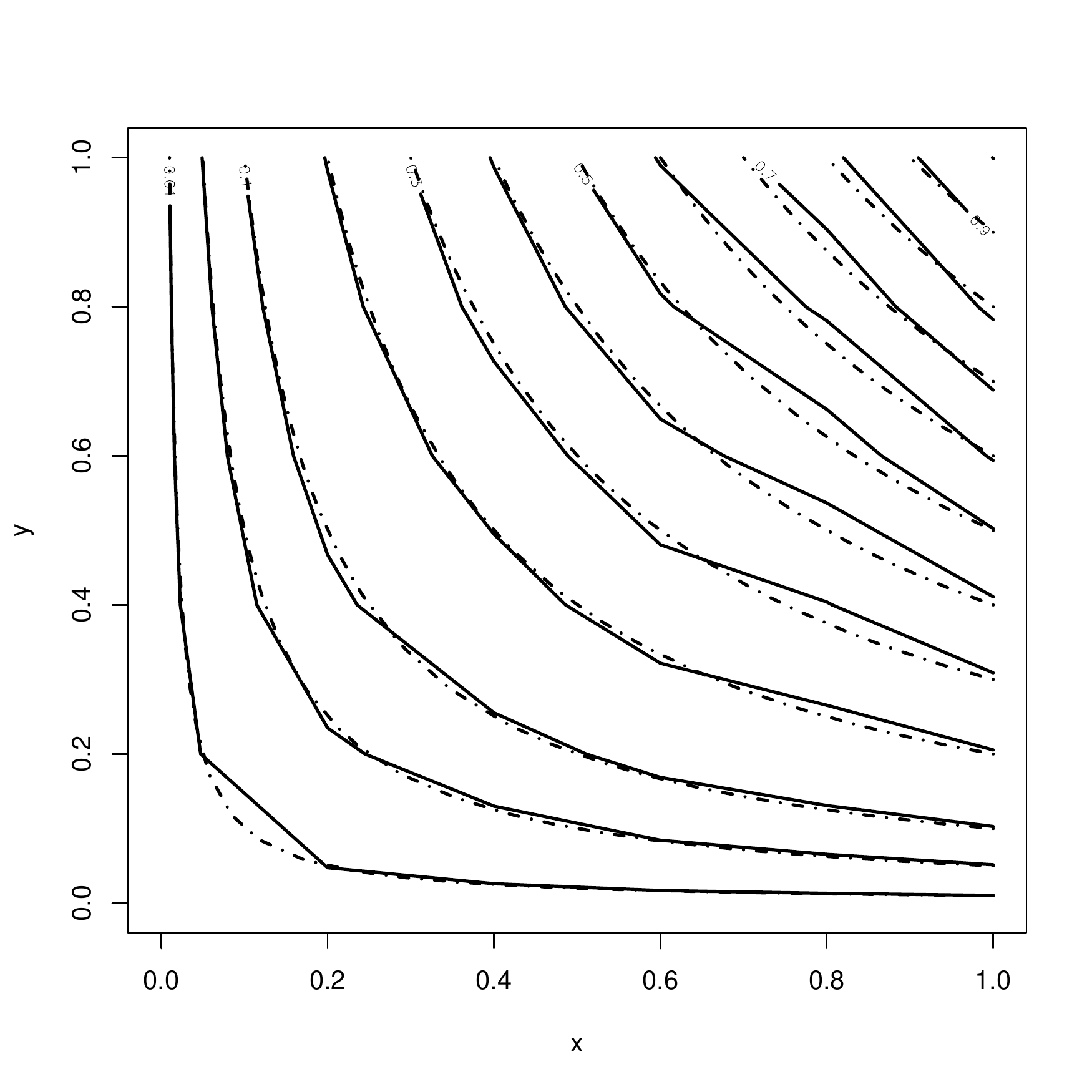} }
\subfigure{
 \includegraphics[width=0.45\textwidth]{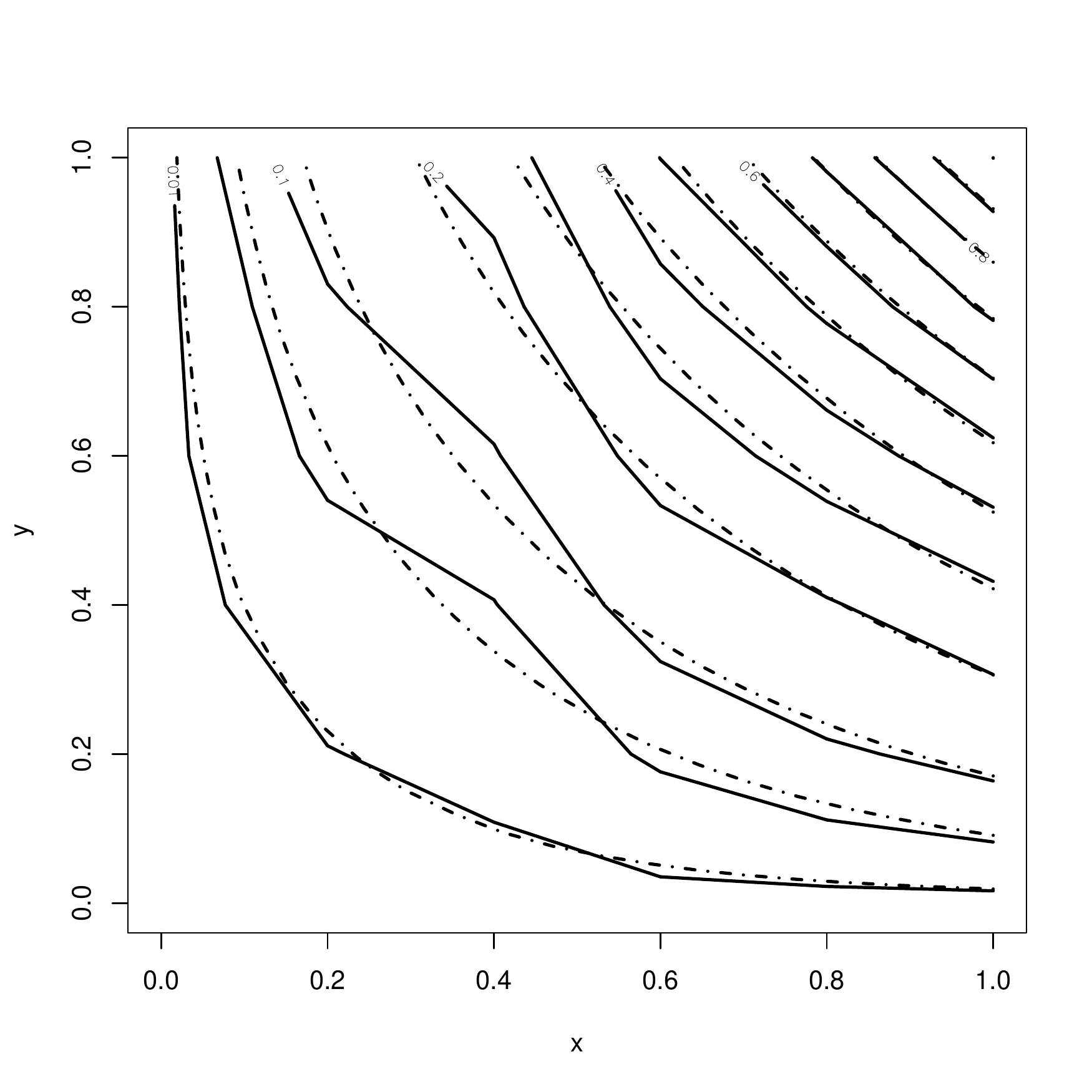} }
\caption{The estimator \fMS\ (upper panels) and contour plot of \FMS\ (lower panels) for two simulations: $f_0(x,y)=1_{[0,1]^2}(x,y)$, $g(t)=1_{[0,1]}(t)$ (left panels) and $f_0(x,y)=x+y$, $g(t) = 2t$ (right panels), $n=5\,000$ and $\dn=\en=0.2$ chosen as illustration. The dash-dotted lines are the true distribution functions $F_0$ for these two examples. The levels of the contour plot are $0.01, 0.05, 0.1, 0.2, \ldots, 1.0$}\label{fig:ex_est}
\end{figure}

\section{Consistency of \FMS\ }\label{sec:cons}
In this section, we prove that $\hMS=h_{\fMS}$ and \FMS\ are consistent estimators for the density $\hf$ of the observable vector $W$ and the bivariate distribution function of interest $F_0$, respectively. To prove this we assume the densities $f_0$ and $g$ to satisfy conditions $(F.1)$ and $(G.1)$ below. We also assume $f_0$ and $g$ are such that \hf\ satisfies the conditions needed for (\ref{P:pos_hnh}) so that existence and uniqueness of \fMS\ are guaranteed with probability converging to one. Furthermore, we assume the binwidths $\dn$ and $\en$ to satisfy condition $(C.1)$ below.
\begin{itemize}
\item[$(F.1)$] The density $f_0$ has compact support $\W_M=[0,M_1]\times[0,M_2]$ and is continuous on $\W_M$.
\item[$(G.1)$] The censoring density $g$ is uniformly continuous and bounded away from zero and infinity on $(0,M_1)$, i.e.\ $0 < c_g \leq g(t) \leq C_g < \infty$ for all $t\in(0,M_1)$.
\item[$(C.1)$] The binwidths $\dn$ and $\en$ converge to zero such that $n\dn\en\to\infty$ as $n\to\infty$.
\end{itemize}
Note that if $\dn \asymp n^{-1/5}$ and $\en \asymp n^{-1/5}$, condition $(C.1)$ is satisfied. The choice of $\dn$ of order $n^{-1/5}$ is probably optimal. One can also choose $\en$ of this order, but it is probably better to choose $\en\ll \dn$, see for a discussion on this matter \shortciteN{witte:11spie}. Further remarks on the problem of binwidth choice can be found in section \ref{sec:discussion}.

\begin{lemma}\label{lem:cons_hMS_cm}
Let $f_0$ and $g$ satisfy conditions $(F.1)$ and $(G.1)$ and $\dn,\ \en$ condition $(C.1)$. Furthermore, assume that $\int_0^{M_1}\log h_0(t)\,dt < \infty$, $\int_0^{M_1}\int_0^{M_1}\log h_1(t,z)\,dz\,dt < \infty$. Then \hMS\ is Hellinger-consistent for $\hf$, i.e.
\bef\label{Hcons_hMS_cm}
\H\big(\hMS,\hf\big) \Pconv 0.
\eef
\end{lemma}

\begin{proof}
We establish (\ref{Hcons_hMS_cm}) using relation (\ref{rel_H_KL}) and the property that \fMS\ minimizes $\K\big(\hn,\hf\big)$. Since $f_0\not\in \F_n$ in general, the inequality
\be \K\big(\hn,h_{\fMS}\big) \leq \K\big(\hn,\hf\big)\ee
need not hold. In order to exploit the defining minimizing property of \fMS, we define a piecewise constant representative $\fn$ of $f_0$ which belongs to $\F_n$ and approximates $f_0$
\be\fn(t,z) = \dn^{-1}\en^{-1}\int_{A_{n,i}} \int_{B_{n,j}}f_0(u,v)\,dv\,du\ \text{if}\ (t,z)\in A_{n,i}\times B_{n,j}.\ee

For this representative it holds that
\bef\label{ineq:HKL_hMS_cm}
0 \leq 2\H\big(\hn, \hMS\big)^2 \leq \K\big(\hn, \hMS\big) \leq \K\big(\hn, \hfn\big),
\eef
also using relation (\ref{rel_H_KL}). The Hellinger distance is a metric, hence applying the triangle inequality twice gives
\bef\label{ineq:triangle_H_hMS_cm}
0 \leq \H\big(\hMS,\hf\big) \leq \H\big(\hMS,\hn\big) + \H\big(\hn, \hfn\big) + \H\big(\hfn,\hf\big).
\eef
The first term on the right hand side of (\ref{ineq:triangle_H_hMS_cm}) converges in probability to zero by combining (\ref{ineq:HKL_hMS_cm}) and Lemma \ref{lem:conv_K_hn_hfn}. The second term converges in probability to zero by combining Lemma \ref{lem:conv_K_hn_hfn} and relation (\ref{rel_H_KL}). The third term converges to zero by combining relation (\ref{rel_H_L1c}) and the second result in Lemma \ref{lem:sup_conv_fn}, hence (\ref{Hcons_hMS_cm}) follows.
\end{proof}

From Lemma \ref{lem:cons_hMS_cm}, it follows that $\FMS$ converges pointwise and in $L_1$-norm to $F_0$.
\begin{theorem}\label{th:cons_cm_Fn}
Under the conditions of Lemma \ref{lem:cons_hMS_cm},
\bef\label{L1conv:FMS}
\norm{\FMS-F_0}{1}\Pconv 0.
\eef
Consequently, for all $(t,z)\in\W_M$
\bef\label{pwconv:FMS}
\FMS(t,z)\Pconv F_0(t,z)
\eef
implying that $\sup_{\W_M}|\FMS(t,z)-F_0(t,z)|\Pconv0$.
\end{theorem}

\begin{remark}\rm
Because $\FMSX(t) = \FMS(t,M_2)$, this lemma implies that $\|\FMSX - F_{0,X}\|_{\infty}\Pconv0$.
\end{remark}

\begin{proof2}{Theorem \ref{th:cons_cm_Fn}}
By combining Lemma \ref{lem:cons_hMS_cm} and relation (\ref{rel_H_L1}), we have that
\begin{align*} \norm{\hMS-\hf}{1} =& \int_{\W_M} \left|\hMS(t,z)-\hf(t,z)\right|\,d\lambda(t,z)\\
=& \int_0^{M_1}\int_0^{M_2}g(t)\big|\partial_2\FMS(t,z)-\partial_2F_0(t,z)\big|\,dz\,dt \\
&+\int_0^{M_1}g(t)\big|(1-\FMSX(t))-(1-F_{0,X}(t))\big|\,dt \leq \sqrt{2}\H\big(\hMS,\hf\big) \Pconv 0.\end{align*}
This implies that
\be \norm{\partial_2\FMS-\partial_2F_0}{1} = \int_{\W_M}\big|\partial_2\FMS(t,z)-\partial_2F_0(t,z)\big|\,d\lambda(t,z) \Pconv 0,\ee
since $g>0$ on $(0,M_1)$. To prove (\ref{L1conv:FMS}) note that the $L_1$-distance between \FMS\ and $F_0$ can be bounded by
\begin{align*} \norm{\FMS-F_0}{1} &= \int_{\W_M} \left|\FMS(t,z)-F_0(t,z)\right|\,d\lambda(t,z)\\
&= \int_{\W_M} \left|\int_0^z\left(\partial_2\FMS(t,v)-\partial_2F_0(t,v)\right)\,dv\right|\,d\lambda(t,z)\\
&\leq \int_{t=0}^{M_1}\int_{v=0}^{M_2}\int_{z=v}^{M_2}\big|\partial_2\FMS(t,v)-\partial_2F_0(t,v)\big|\,dz\,dv\,dt\\
&\leq M_2\norm{\partial_2\FMS-\partial_2F_0}{1} \Pconv 0.\end{align*}

To prove (\ref{pwconv:FMS}), assume it does not hold for a certain $(t,z)\in\W_M$, i.e. there exists $\eps>0$, $\delta>0$ and a subsequence $n_k$ of $n$ such that for all $k\in\NN$
\be P\left(\big|\hat F_{n_k}^{MS}(t,z) - F_0(t,z)\big| \geq \delta\right) \geq \eps.\ee
Assume $\hat F_{n_k}^{MS}(t,z) \leq F_0(t,z) - \delta$, then there exists a small $c>0$ such that
\be \forall\ (u,v) \in [t-c\delta,t]\times[z-c\delta,z] =: \A_{\delta} \ :\ \hat F_{n_k}^{MS}(u,v) \leq F_0(u,v) - \frac12\delta,\ee
by continuity of $F_0$ and monotonicity of \FMS. This implies that for all $k$
\be \lefteqn{P\left(\int_{\W_M} \big| \hat F_{n_k}^{MS}(t,z) - F_0(t,z)\big|\,dz\,dt \geq \frac{c^2}{2}\delta\right)}\\
&&\geq P\left(\int_{\A_{\delta}} \big|\hat F_{n_k}^{MS}(u,z) - F_0(u,z)\big|\,dz\,du \geq \frac{c^2}{2}\delta\right)\\
&&\geq P\left(\forall\ (u,v)\in\A_{\delta} : \hat F_{n_k}^{MS}(u,v) \leq F_0(u,v) - \frac12\delta\right) \geq \eps.
\ee
If $\hat F_{n_k}^{MS}(t,z) \geq F_0(t,z) + \delta$, we have by a similar argument that
\be \forall\ (u,v) \in [t,t+c\delta]\times[z,z+c\delta] =: \A_{\delta} \ :\ \hat F_{n_k}^{MS}(u,v) \geq F_0(u,v) + \frac12\delta,\ee
giving that for all $k$
\be \lefteqn{P\left(\int_{\W_M} \big| \hat F_{n_k}^{MS}(t,z) - F_0(t,z)\big|\,dz\,dt \geq \frac{c^2}{2}\delta\right)}\\
&&\geq P\left(\forall\ (u,v)\in\A_{\delta} : \hat F_{n_k}^{MS}(u,v) \geq F_0(u,v) + \frac12\delta\right) \geq \eps.\ee
This contradicts (\ref{L1conv:FMS}). Strengthening pointwise consistency to uniform consistency over $\W_M$ follows from monotonicity of $ \hat F_{n_k}^{MS}$ and $F_0$ and the assumed smoothness of $F_0$.
\end{proof2}

\section{Discussion}\label{sec:discussion}
In this paper we have considered consistency of the MSLE, where the observation distribution was smoothed by using histogram type estimators. Rigorous derivation of the asymptotic distribution is at this moment still not available. Heuristic considerations indicate that, if $(t_0,z_0)$ is an interior point of the support of $f_0$, and the binwidth for the first coordinate satisfies $\dn\sim c_1n^{-1/5}$, whereas the binwidth for the second coordinate satisfies $n^{-2/5}\ll\en\ll n^{-1/5}$, we get
\be n^{2/5}\big(\FMS(t_0,z_0)-F_0(t_0,z_0)\big)\Dconv\N(\beta,\sigma^2),\ee
where $\N(\beta,\sigma^2)$ is a normal distribution with expectation
\be \beta=\frac{\partial_1F_0(t_0,z_0)g'(t_0)c_1^2}{6g(t_0)},\ee
and variance
\be \sigma^2=\frac{F_0(t_0,z_0)\big(1-F_0(t_0,z_0)\big)\sqrt{3}}{2c_1g(t_0)}.\ee

This implies that the asymptotically optimal number of cells for the first coordinate would satisfy
\be\dn^{-1}\sim 3^{-1/2}\left(\frac{2 g'(t_0)^2\partial_1F_0(t_0,z_0)^2}{g(t_0)F_0(t_0,z_0)\big(1-F_0(t_0,z_0)\big)}\,n\right)^{1/5}.\ee
implying that the optimal number of cells on the first coordinate is rather small for the model on which the simulations, reported below, are based.

\begin{figure}[!ht]
\centering
\subfigure{
 \includegraphics[width=0.45\textwidth]{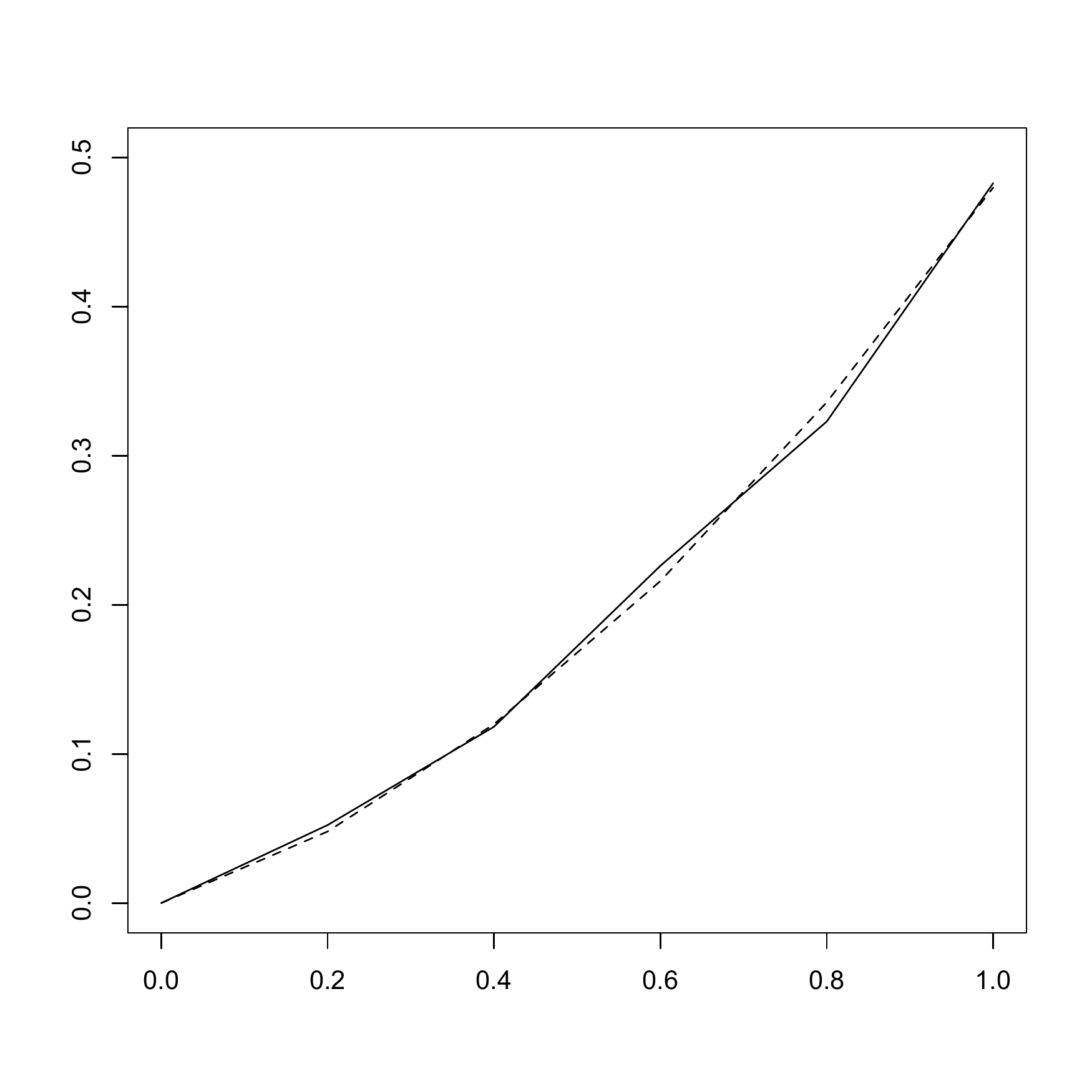}}
\subfigure{
 \includegraphics[width=0.45\textwidth]{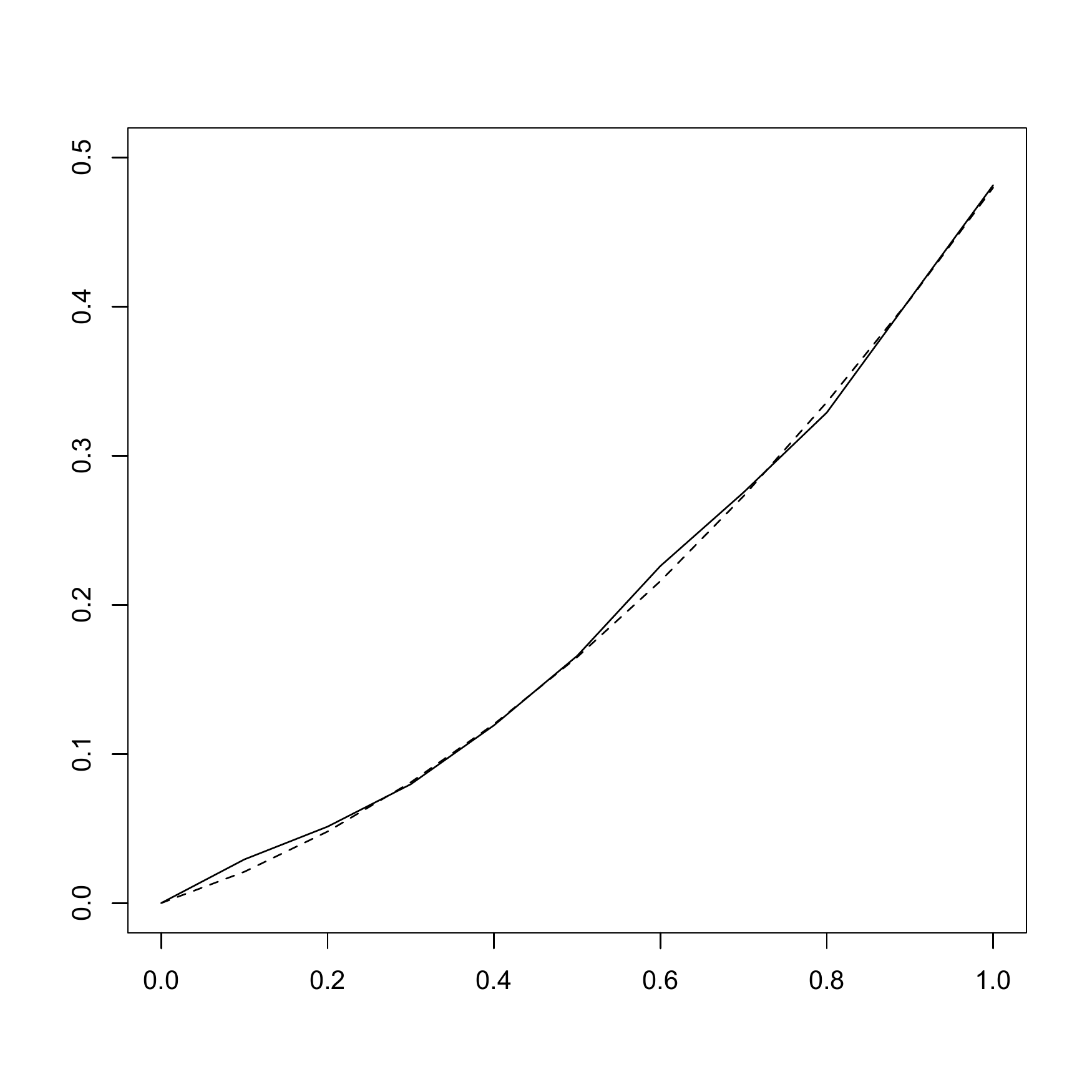}}
\caption{Estimates of the function $t\mapsto F_0(t,z)$, where $z=0.6$. The MSLE is shown in the left panel and the plug-in estimate $F_n$, defined by (\ref{plugin_est}), in the right panel. The MSLE and $F_n$ are the piecewise linear solid curves in the pictures and the dashed curves represent the real $F_0$, where $F_n$ is linearly extended to the last interval (where it can not be defined by interpolation between values at successive points of the grid). Moreover, $F_0(t,z)=\tfrac12 tz(t+z)$, $g(t)=2t$, and the sample size for which the estimators were computed was $n=5000$. The binwidth for the first coordinate was $0.2$ for the MSLE and $0.1$ for the plug-in estimator. For the second coordinate we took binwidth $0.2$ for both estimators.}
\label{fig:F(t,0.6)}
\end{figure}

The behavior of the MSLE $\FMS$ is somewhat similar to that of the plug-in estimator $F_n$, defined by
\bef\label{plugin_est}
F_n(a_{n,i},b_{n,j})=\frac{\int_{t\in A_{n,i}\cup A_{n,i+1},\,z\in(0,b_{n,j}]}\,d\Hn{n}(t,z)}{\int_{t\in A_{n,i}\cup A_{n,i+1}}\,d\Gn{n}(t)},
\eef
at the points $(a_{n,i},b_{n,j})$ of the grid, and by linear interpolation elsewhere (except on the last interval $A_{n,k_n}$, where it was just linearly extended), where $\Gn{n}$ is the empirical distribution function of the observations $T_1,\dots,T_n$. However $F_n$ and $\FMS$ are not asymptotically equivalent, as first was noticed in the simulations. We have, as $(a_{n,i},b_{n,j})\to(t_0,z_0)$, under the same conditions on the binwidth $\dn$ and $\en$ as used above,
\be n^{2/5}\big(F_n(a_{n,i},b_{n,j})-F_0(a_{n,i},b_{n,j})\big)\Dconv\N(\beta_2,\sigma_2^2),\ee
where
\be \beta_2=\left\{\frac16\partial_1^2F_0(t_0,z_0)+\frac{\partial_1F_0(t_0,z_0)g'(t_0)}{3g(t_0)}\right\}c_1^2, \qquad \sigma_2^2=\frac{F_0(t_0,z_0)\big(1-F_0(t_0,z_0)\big)}{2c_1g(t_0)}.\ee
This implies that the asymptotic variance is smaller by a factor $\sqrt{3}$ than the conjectured asymptotic variance of the MSLE. On the other hand, the asymptotic bias is larger than the conjectured bias of the MSLE $\FMS$ in the model, used in the simulations which produced Table \ref{table:simstudy}. It seems unavoidable that the relation between plug-in estimators of this type and our MSLE $\FMS$ involves the partial derivative $\partial_2\FMS$, which makes the analysis rather complicated. 

Other smoothing methods are also possible, for example using kernel estimators instead of histogram estimators for the smoothing of the observation distribution. However, we do not know how to compute the MSLE for this type of smoothing. Using a smoothed MLE (SMLE) is not sensible because it inherits the inconsistency of the unsmoothed MLE.

In Table \ref{table:simstudy} we compare the local mean squared error (MSE) of the MSLE with the MSE's of other comparable estimators. On the second coordinate we took $5$ cells for the MSLE, which means that the bias on the second coordinate does not play a role, since $0.6$ is then a point of the grid for the second coordinate, and on the first coordinate we took the number of cells between 4 (for $n=500$) and 7 (for $n=10\,000$). The results were obtained by generating $10\,000$ samples for each value of $(t,z)$, considered in the table, and each sample size $n$. We compared the results with the MSE's of the plug-in estimator $\Fpi{1}$, studied in \shortciteN{witte:11spie}, and defined by
\bef\label{F_n1}
\Fpi{1}(t_0,z_0)=\frac{\int_{z\in(0,z_0]}k_{\dn}(t_0-u)\,d\Hn{n}(u,z)}{\int k_{\dn}(t_0-u)\,d\Gn{n}(u)},\qquad k_{\dn}(u)=\dn^{-1}k(u/\dn),
\eef
where $k$ is a smooth symmetric kernel with support $[-1,1]$, for example the Epanechnikov kernel, and $\dn$ the bandwidth. Note the similarity between $(\ref{F_n1})$ and (\ref{plugin_est}). We also included the binned MLE of \citeN{maathuis:08} in our comparison. The values for $\Fpi{1}$ and the binned MLE were taken from Table 5.1, p.\ 115, \citeN{witte:11}, where the bandwidths, resp.\ binwidth, were chosen in such a way that the MSE was minimized. As can be seen from the table, none of the four estimators comes out as uniformly best in this situation.

\begin{table}[!ht]
\centering
\caption{Estimated values of the MSE for four estimators of $F_0(t_0,0.6)$ at a number of values of $t_0$. The boldfaced values in each row are the minimal values of the MSE in that row.}
\label{table:simstudy}
\vspace{0.5cm}
\begin{tabular*}{\textwidth}{@{\extracolsep{\fill}}lrcccc}
\hline
 & \\[-4mm]
 & $n$ & MSLE \FMS & Plug-in $F_n$ & $\Fpi{1}$ & binned MLE\\
 & \\[-5mm]
\hline
 & \\[-4mm]
$t_0=0.2$ & $500$ & $2.12\times10^{-3}$ & $1.41\times10^{-3}$ & $2.81\times10^{-3}$ &$\mathbf{7.84\times10^{-4}}$ \\
 & $1\,000$ & $1.86\times 10^{-3}$ & $7.73\times 10^{-4}$ & $1.53\times10^{-3}$ & $\mathbf{2.01\times10^{-4}}$\\
 & $5\,000$ & $3.19\times 10^{-4}$ & $1.96\times 10^{-4}$ & $2.04\times10^{-4}$ & $\mathbf{1.49\times 10^{-4}}$\\
 & $10\,000$ & $1.35\times 10^{-4}$ & $\mathbf{1.11\times 10^{-4}}$ & $9.59\times10^{-5}$ & $1.13\times 10^{-4}$\\
 & \\[-4mm]
\hline
 & \\[-5mm]
$t_0=0.4$ & $500$ & $\mathbf{8.39\times10^{-4}}$ & $1.25\times 10^{-3}$ & $9.07\times 10^{-4}$ &$1.21\times 10^{-3}$\\
 & $1\,000$ & $\mathbf{4.90\times10^{-4}}$ & $7.07\times 10^{-4}$ & $5.94\times 10^{-4}$ &$6.74\times 10^{-4}$\\
 & $5\,000$ & $\mathbf{1.21\times10^{-4}}$ & $1.90\times10^{-4}$ & $1.32\times 10^{-4}$ &$2.37\times 10^{-4}$\\
 & $10\,000$ & $\mathbf{8.35\times10^{-5}}$ & $1.08\times10^{-4}$ & $8.95\times 10^{-5}$ &$1.35\times 10^{-4}$\\
 & \\[-4mm]
\hline
 & \\[-5mm]
$t_0=0.6$ & $500$ & $\mathbf{6.32\times10^{-4}}$ &  $1.17\times10^{-3}$ & $8.21\times 10^{-4}$ &$1.38\times 10^{-3}$\\
 & $1\,000$ & $\mathbf{3.71\times10^{-4}}$ & $6.86\times10^{-4}$ & $5.31\times10^{-4}$ &$7.79\times 10^{-4}$\\
 & $5\,000$ & $1.48\times10^{-4}$ & $1.86\times10^{-4}$ & $\mathbf{1.21\times10^{-4}}$ &$2.11\times 10^{-4}$\\
 & $10\,000$ & $\mathbf{7.80\times10^{-5}}$ & $1.06\times 10^{-4}$ & $9.21\times 10^{-5}$ &$1.31\times 10^{-4}$\\
 & \\[-4mm]
\hline
 & \\[-5mm]
$t_0=0.8$ & $500$ & $6.71\times10^{-4}$ &  $9.43\times10^{-4}$ & $\mathbf{5.91\times10^{-4}}$ &$1.39\times 10^{-3}$\\
 & $1\,000$ & $5.88\times10^{-4}$ & $5.85\times10^{-4}$ & $\mathbf{3.14\times10^{-4}}$ &$8.59\times 10^{-4}$\\
 & $5\,000$ & $9.65\times10^{-5}$ & $1.81\times10^{-4}$ & $\mathbf{5.61\times10^{-5}}$ &$2.27\times 10^{-4}$\\
 & $10\,000$ & $5.84\times10^{-5}$ & $1.04\times10^{-4}$ & $\mathbf{3.25\times10^{-5}}$ &$1.39\times 10^{-4}$\\
 & \\[-4mm]
\hline
\end{tabular*}
\end{table}

\appendix
\section{Technical lemmas and proofs}\label{app:technlem}
In this section, we prove most of the results stated in the previous sections as well as some technical lemmas needed in these proofs. We start with some known results on several distances.

Let $f$ and $g$ be two probability densities with respect to a dominating measure $\mu$. Let \H\ and \K\ denote the Hellinger distance and the Kullback-Leibler divergence between $f$ and $g$ respectively, i.e.
\be \H\big(f,g\big) = \sqrt{\frac{1}{2}\int\big(\sqrt{f(x)}-\sqrt{g(x)}\big)^2\,d\mu(x)}, \qquad \K\big(f,g\big) = \int f(x)\log\frac{f(x)}{g(x)}\,d\mu(x).\ee
Between \H, \K\ and the $L_1$-norm $\norm{\cdot}{1}$ we use the following relations
\bef
\lefteqn{2\H\big(f,g\big)^2 \leq \K\big(f,g\big),} \label{rel_H_KL}\\
\lefteqn{\H\big(f,g\big)^2 \leq \frac12\norm{f-g}{1} \leq \sqrt{2}\H\big(f,g\big),} \label{rel_H_L1}
\eef
see e.g., \citeN{geer:00} Lemma 1.3 for (\ref{rel_H_KL}) and \citeN{lecam:86} p.\ 47 for (\ref{rel_H_L1}). If $f$ and $g$ have compact support $\C$ with finite measure $\mu(\C)=C<\infty$, then
\bef\label{rel_H_L1c}
\H\big(f,g\big)^2 \leq \norm{f-g}{1} = \int_{\C}\big|f(x)-g(x)\big|\,d\mu(x) \leq C\norminfty{f-g}.
\eef

Now, we can turn to the proofs and technical lemmas.

\begin{proof2}{Lemma \ref{lem:rewr_lS_cm}}
Let $f\in \F_n$, then for $t\in A_{n,i}=(a_{n,i-1},a_{n,i}],\ z\in B_{n,j}=(b_{n,j-1},b_{n,j}]$ we can write
\be \lefteqn{1-F_X(t) = \int_{u=t}^{M_1}\int_{z=0}^{M_2} f(u,z)\,dz\,du=(a_{n,i}-t)\sum_{j=1}^{l_n}\en f_{i,j}+\sum_{l=i+1}^{k_n}\sum_{j=1}^{l_n}\dn\en f_{l,j}}\\
\lefteqn{\partial_2 F(t,z) = \int_{u=0}^t f(u,z)\,du = \sum_{l=1}^{i-1} \dn f_{l,j} + (t-a_{n,i-1})f_{i,j},}\ee
so that
\be \lefteqn{\int_{A_{n,i}}\log\big(1-F_X(t)\big)\,dt = \int_{A_{n,i}}\log\left\{\sum_{j=1}^{l_n}\sum_{l=i+1}^{k_n}\dn\en f_{l,j} + a_{n,i}\sum_{j=1}^{l_n}\en f_{i,j} - t \sum_{j=1}^{l_n}\en f_{i,j}\right\},} \\
\lefteqn{\int_{A_{n,i}}\int_{B_{n,j}}\log \partial_2 F(t,z)\,dz\,dt = \en\int_{A_{n,i}}\log\left\{\sum_{l=1}^{i-1} \dn f_{l,j} - a_{n,i-1}f_{i,j} + f_{i,j}t\right\}\,dt.}\ee
Since we have for $0\leq a<b<\infty$, $\sigma\not=0$ and $\tau \geq -\sigma a$
\begin{align*} \int_a^b\log(\tau+\sigma t)\, dt &= \frac{1}{\sigma}\big[u\log u - u\big]_{u=\tau+\sigma a}^{\tau+\sigma b}\\
&=\frac{1}{\sigma}(\tau+\sigma b)\log\big(\tau+\sigma b\big) - \frac{1}{\sigma}(\tau+\sigma a)\log\big(\tau+\sigma a\big) - (b-a),\end{align*}
we get
\be \lefteqn{\int_{A_{n,i}}\log\left\{\sum_{j=1}^{l_n}\sum_{l=i+1}^{k_n}\dn\en f_{l,j} + a_{n,i}\sum_{j=1}^{l_n}\en f_{i,j} - t \sum_{j=1}^{l_n}\en f_{i,j}\right\}\,dt}\\
&&= -\frac{1}{\en\sum_{j=1}^{l_n}f_{i,j}} \left\{\left(\dn\en \sum_{l=i+1}^{k_n}\sum_{j=1}^{l_n}f_{l,j}\right)\log\left(\dn\en \sum_{l=i+1}^{k_n}\sum_{j=1}^{l_n}f_{l,j}\right)\right.\\
&&\phantom{-\frac{1}{\en\sum_{j=1}^{l_n}f_{i,j}} \left\{\right.} -\left.\left(\dn\en \sum_{l=i}^{k_n}\sum_{j=1}^{l_n}f_{l,j}\right)\log\left(\dn\en \sum_{l=i}^{k_n}\sum_{j=1}^{l_n}f_{l,j}\right)\right\}-\dn\ee
and
\be \lefteqn{\en\int_{A_{n,i}}\log\left\{\sum_{l=1}^{i-1} \dn f_{l,j} - a_{n,i-1}f_{i,j} + f_{i,j}t\right\}\,dt}\\
&&=\en\frac{1}{f_{i,j}}\left\{\left(\sum_{l=1}^{i} \dn f_{l,j}\right)\log\left(\sum_{l=1}^{i} \dn f_{l,j}\right)-\left(\sum_{l=1}^{i-1} \dn f_{l,j}\right)\log\left(\sum_{l=1}^{i-1} \dn f_{l,j} \right)\right\} - \en\dn,\ee
so that
\begin{align*} l^S(f) =& -\sum_{i=1}^{k_n}\frac{\hi}{\en\sum_{j=1}^{l_n}f_{i,j}}\left\{\left(\dn\en \sum_{l=i+1}^{k_n}\sum_{j=1}^{l_n}f_{l,j}\right)\log\left(\dn\en \sum_{l=i+1}^{k_n}\sum_{j=1}^{l_n}f_{l,j}\right)\right.\\
&\phantom{-\sum_{i=1}^{k_n}\frac{\hi}{\en\sum_{j=1}^{l_n}f_{i,j}}\left\{\right.}
-\left.\left(\dn\en \sum_{l=i}^{k_n}\sum_{j=1}^{l_n}f_{l,j}\right)\log\left(\dn\en \sum_{l=i}^{k_n}\sum_{j=1}^{l_n}f_{l,j}\right)\right\}\\
& +\sum_{i=1}^{k_n}\sum_{j=1}^{l_n}\frac{\en\hij}{f_{i,j}}\left\{\left(\dn\sum_{l=1}^{i} f_{l,j}\right)\log\left(\dn\sum_{l=1}^{i} f_{l,j}\right)-\left(\dn\sum_{l=1}^{i-1} f_{l,j}\right)\log\left(\dn\sum_{l=1}^{i-1} f_{l,j} \right)\right\} \\
& - \sum_{i=1}^{k_n}\dn\hi - \sum_{i=1}^{k_n}\sum_{j=1}^{l_n}\dn\en\hij. \end{align*}
The last two terms can be left out in the maximization, since they do not depend on $f$. Now, taking $\ai(f)$ and $\bij(f)$ as in (\ref{def:alpha_beta}) we have that
\be \en\sum_{j=1}^{l_n}f_{i,j} = -\frac{1}{\dn}(\alpha_{i+1}(f) - \ai(f)), \qquad f_{i,j} = \frac{1}{\dn}\left(\bij(f)-\beta_{i-1,j}(f)\right),\ee
so that $l^S(f) = \psi(f)$ for
\be \psi(f) = \dn\sum_{i=1}^{k_n} \hi\ \phi\big(\alpha_{i+1}(f),\ai(f)\big)+ \dn\en\sum_{i=1}^{k_n}\sum_{j=1}^{l_n} \hij\ \phi\big(\bij(f),\beta_{i-1,j}(f)\big)\ee
with $\phi$ as defined in (\ref{def:phixy}), hence the maximizer of $l^S(f)$ is the maximizer of $\psi(f)$.

The estimator \fMS\ has to satisfy the following conditions
\be (S.1) && \dn\en\sum_{i=1}^{k_n}\sum_{j=1}^{l_n}f_{i,j} = 1,\\
(S.2) && \forall\ i,j\ :\ f_{i,j} \geq 0. \ee
To get condition $(S.1)$ in the objective function $\psi(f)$, we include a Langrange multiplier $\lambda\in\R$ and maximize
\be \psi_{\lambda}(f) = \psi(f) + \lambda\left(\dn\en\sum_{i=1}^{k_n}\sum_{j=1}^{l_n}f_{i,j}-1\right)\ee
over $\cP_n = \big\{f\in\R^{k_nl_n}: f_{i,j} \geq 0 \ \forall i,j\big\}$. For any function $f$ satisfying condition $(S.1)$ and each $\lambda$, $\psi_{\lambda}(f)$ equals $\psi(f)$ and for $\hat f_{\lambda} = \arg\max_{f\in\cP_n}\psi_{\lambda}(f)$ we have
\be 0 = \lim_{\gamma\to0}\gamma^{-1}\Big(\psi_{\lambda}\big((1+\gamma)\hat f_{\lambda}\big) - \psi_{\lambda}\big(\hat f_{\lambda}\big)\Big) = 1 + \lambda\dn\en\sum_{i=1}^{k_n}\sum_{j=1}^{l_n}\hat f_{\lambda,i,j}.\ee
This means that if we take $\lambda=-1$, the global maximizer of $\psi_{\lambda}(f)$ over $\cP_n$ is contained in $\F_n$, so that
\be \arg\max_{f\in\cP_n}\psi_{-1}(f) 
= \arg\max_{f\in\F_n} \psi(f).\ee 
Since we also have that $\F_n\subset\B_n$, it follows that $\arg\max_{f\in\cP_n}\psi_{-1}(f)\in\B_n$, hence
\be \arg\max_{f\in\F_n}l^S(f) = \arg\max_{f\in\B_n}\psi_{-1}(f).\ee
\end{proof2}

The piecewise constant representative $\fn$ and the corresponding $\hfn$ converge to the true $f_0$ and $\hf$ under condition $(C.1)$. This is stated in Lemma \ref{lem:sup_conv_fn} below.
\begin{lemma}\label{lem:sup_conv_fn}
Let $f_0$ and $g$ satisfy conditions $(F.1)$ and $(G.1)$ and $\dn$, $\en$ condition $(C.1)$, then
\be \norminfty{\fn - f_0} \conv 0, \qquad \norminfty{\hfn-\hf} \conv 0.\ee
\end{lemma}

\begin{proof}
For the proof of the first result, note that for $(t,z)\in A_{n,i}\times B_{n,j}$
\begin{align*} |\fn(t,z) - f_0(t,z)| &= \left|\dn^{-1}\en^{-1}\int_{A_{n,i}}\int_{B_{n,j}} f_0(u,v)\,du\,dv -f_0(t,z)\right|\\
&\leq \dn^{-1}\en^{-1}\int_{A_{n,i}}\int_{B_{n,j}} \big|f_0(u,v)-f_0(t,z)\big|\,du\,dv\\
&\leq\max_{(u,v)\in A_{n,i}\times B_{n,j}} |f_0(u,v)-f_0(t,z)|.\end{align*}
Using $(C.1)$ and the uniform continuity of $f_0$ on $\W_M$ we get for $(t,z)\in\W_M$
\be |\fn(t,z) - f_0(t,z)| \leq \max_{i,j}\max_{(u,v)\in A_{n,i}\times B_{n,j}} |f_0(u,v)-f_0(t,z)| \conv 0,\ee
uniformly in $(t,z)$ as $n\to\infty$. This implies the first result.

For the proof of the second result, note that for $(t,z)\in\W_M$
\begin{align*} |\hfn(t,z) - \hf(t,z)| =& \left|g(t)\left\{1_{\{z>0\}}\int_{u=0}^t\big\{\fn(u,z) - f_0(u,z)\big\}\,du \right.\right.\\
&\phantom{\left|g(t)\left\{\right.\right.}\left.\left. + 1_{\{z=0\}}\int_{u=t}^{M_1}\int_{z=0}^{M_2}\big\{\fn(u,z) - f_0(u,z)\big\}\,dz\,du\right\}\right|\\
\leq& \norminfty{g}\left\{ M_1\norminfty{\fn - f_0} + M_1M_2\norminfty{\fn - f_0} \right\}.\end{align*}
Since this upper bound does not depend on $t$ and $z$, the second result now follows from the first.
\end{proof}

\begin{lemma}
Define for $(t,z)\in A_{n,i}\times B_{n,j}$
\be \bar h_n(t,z) = \E\, \hn(t,z) =  \dn^{-1}\en^{-1}\int_{A_{n,i}}\int_{B_{n,j}}\hf(u,v)\,d\lambda(u,v),\ee
then, under the conditions of Lemma \ref{lem:sup_conv_fn},
\bef\label{univ_conv:hfn_bhn}
\norminfty{\hfn-\bar h_n} \conv 0.
\eef
\end{lemma}

\begin{proof}
First, note that
\be \norminfty{\hfn-\bar h_n} \leq \norminfty{\hfn-\hf} + \norminfty{\hf-\bar h_n}.\ee
The first term converges to zero by Lemma \ref{lem:sup_conv_fn}. We now prove that the second term also converges to zero. To see this, note that
\be \norminfty{\hf-\bar h_n} \leq \norminfty{h_1-\bar h_{1,n}} + \norminfty{h_0-\bar h_{0,n}}.\ee
Similarly as the first result in Lemma \ref{lem:sup_conv_fn}, we have for $(t,z)\in A_{n,i}\times B_{n,j}$
\begin{align*} \big|h_1(t,z) - \bar h_{1,n}(t,z)\big| &= \left|\dn^{-1}\en^{-1}\int_{A_{n,i}}\int_{B_{n,j}} \big(h_1(t,z) - h_1(u,v)\big)\,dv\,du\right|\\
&\leq \dn^{-1}\en^{-1}\int_{A_{n,i}}\int_{B_{n,j}} \big|h_1(t,z) - h_1(u,v)\big|\,dv\,du\\
&\leq \max_{(u,v)\in A_{n,i}\times B_{n,j}}\big|h_1(t,z) - h_1(u,v)\big|.\end{align*}
Both $g$ and $\partial_2 F_0$ are uniformly continuous, hence $h_1$ is as well and with condition $(C.1)$ we get for $(t,z)\in\W_M$
\be \big|h_1(t,z) - \bar h_{1,n}(t,z)\big| \leq \max_{i,j}\max_{(u,v)\in A_{n,i}\times B_{n,j}}\big|h_1(t,z) - h_1(u,v)\big| \conv 0,\ee
uniformly in $(t,z)$ as $n\to\infty$. Via a similar argument we get that
\be \big|h_0(t) - \bar h_{0,n}(t)\big| \leq \max_i\max_{u\in A_{n,i}}\big|h_0(t) - h_0(u)\big| \conv 0,\ee
uniformly in $t$ as $n\to\infty$, hence
$\norminfty{\hf-\bar h_n}\conv 0$.
\end{proof}

\begin{lemma}\label{lem:conv_K_hn_hfn}
Under the conditions of Lemma \ref{lem:cons_hMS_cm} such that \hf\ satisfies (\ref{cond:h}),
\bef\label{conv:K_hn_hfn}
\K\big(\hn,\hfn\big) \Pconv 0.
\eef
\end{lemma}

\begin{proof}
We can write
\begin{align}\label{rewr:K_hn_hfn}
\K\big(\hn, \hfn\big) =& \int_{\W_M}\hn(t,z) \log\frac{\hn(t,z)}{\hfn(t,z)}\,d\lambda(t,z)\nonumber\\
=& \int_{\W_M}\hn(t,z) \log\frac{\hn(t,z)}{\hf(t,z)}\,d\lambda(t,z) - \int_{\W_M}\hn(t,z) \log\frac{\bar h_n(t,z)}{\hf(t,z)}\,d\lambda(t,z) \\
& -\int_{\W_M}\hn(t,z) \log\frac{\hfn(t,z)}{\bar h_n(t,z)}\,d\lambda(t,z).\nonumber
\end{align}
The expectation of the first term converges to zero by \shortciteN{barron:92} Theorem 5 with $\mu = H_{f_0}$ (the distribution function of the observable vector $W$), $\nu=\lambda$ as defined in (\ref{def:lambda}) and $\cP = \big\{[0,M_1]\times \{0\}, [0,M_1]\times(0,M_2]\big\}$.\

By Fubini's theorem, the expectation of the second term equals
\be \int_{\W_M}\E\,\hn(t,z) \log\frac{\bar h_n(t,z)}{\hf(t,z)}\,d\lambda(t,z) = \int_{\W_M}\bar h_n(t,z) \log\frac{\bar h_n(t,z)}{\hf(t,z)}\,d\lambda(t,z) = \K\big(\bar h_n,\hf\big).\ee
This converges to zero by \shortciteN{barron:92} Theorem 4, so also the expectation of the second term in (\ref{rewr:K_hn_hfn}) converges to zero.

By (\ref{cond:h}), $\bar h_n(t,z) \geq c_{\bar h} >0$ for all $(t,z)\in\W_M$, so that (\ref{univ_conv:hfn_bhn}) implies that for any $\eps>0$ and $n$ sufficiently large
\be 1-\frac{\eps}{c_{\bar h}} \leq 1+\frac{\hfn(t,z)-\bar h_n(t,z)}{\bar h_n(t,z)} \leq 1+\frac{\eps}{c_{\bar h}}.\ee
Then
\begin{align*} \log\left(1-\frac{\eps}{c_{\bar h}}\right) &= \int_{\W_M}\bar h_n(t,z)\log\left(1-\frac{\eps}{c_{\bar h}}\right)\,d\lambda(t,z)\\
&\leq \int_{\W_M} \bar h_n(t,z)\log\left(1+\frac{\hfn(t,z)-\bar h_n(t,z)}{\bar h_n(t,z)}\right)\,d\lambda(t,z) \\
&\leq \int_{\W_M}\bar h_n(t,z)\log\left(1+\frac{\eps}{c_{\bar h}}\right)\,d\lambda(t,z) = \log\left(1+\frac{\eps}{c_{\bar h}}\right),\end{align*}
so that also the expectation of the third term in (\ref{rewr:K_hn_hfn}) converges to zero. Therefore, the expectation of $\K\big(\hn, \hfn\big)$ converges to zero, and because $\K\big(\hn, \hfn\big)\geq0$ a.s.\ the convergence in (\ref{conv:K_hn_hfn}) now follows.
\end{proof}

\section{The EM algorithm}
\label{sec:EM}
Let, as before, $\hat H_n$ denote the smoothed $\Hn{n}$, using the histograms on the rectangles $R_{i,j}\stackrel{\mbox{\small def}}=A_{n,i}\times B_{n,j}$ of the grid. The MSLE has to maximize
\be\int_{z>0}\log\partial_2F(t,z)\,d\hat H_n(t,z)+\int \log\big(1-F_X(t)\big)\,d\hat H_n(t,0),\ee
where
\be\lefteqn{\partial_2 F(t,z)=\en^{-1}\left\{\sum_{l:a_{n,l}<t}f_{l,j}+\frac{t-a_{n,i-1}}{\dn}f_{i,j}\right\},\ \text{if}\ (t,z)\in R_{i,j},}\\
\lefteqn{1-F_X(t)=\sum_{l:a_{n,l-1}>t}\sum_{j=1}^{l_n}f_{l,j}+\frac{a_{n,i}-t}{\dn}\sum_{j=1}^{l_n}f_{i,j}, \ \text{if}\ t\in A_{n,i},}\ee
and $\sum_{i,j} f_{i,j}=1$. Note that we do not parametrize by the densities, but by the total mass $f_{i,j}$ of the distribution on a cell $R_{i,j}=A_{n,i}\times B_{n,j}$. This amounts to the same for this model, however.

The $E$-step, if $t\in A_{n,i}$ and $z\in B_{n,j}$, and $z>0$, is given by
\be\lefteqn{\E\,\big(\log f(X,Z)|T=t,\,Z=z\big)}\\
&&=\sum_{k<i}\frac{f_{k,j}^{(m)}}{\sum_{l:a_{n,l}<t}f_{l,j}^{(m)}+\frac{t-a_{n,i-1}}{\dn}f_{i,j}^{(m)}}\,\log f_{k,j}+\frac{\frac{t-a_{n,i-1}}{\dn}f_{i,j}^{(m)}}{\sum_{l:a_{n,l}<t}f_{l,j}^{(m)}+\frac{t-a_{n,i-1}}{\dn}f_{i,j}^{(m)}}\,\log f_{i,j}.\ee
after the $m$th iteration, and if $z=0$ we get after the $m$th iteration
\be \lefteqn{\E\,\big(\log f(X,Z)|T=t,\,Z=0\big)}\\
&&=\sum_{k>i}\frac{F_k^{(m)}}{\sum_{l:a_{n,l-1}>t}F_l^{(m)}+\frac{a_{n,i}-t}{\dn}F_{i}^{(m)}}\,\log F_k+\frac{\frac{a_{n,i}-t}{\dn}F_i^{(m)}}{\sum_{l:a_{n,l-1}>t}F_l^{(m)}+\frac{a_{n,i}-t}{\dn}F_{i}^{(m)}}\,\log F_i,\ee
where $F_i=\sum_{j=1}^{l_n}f_{i,j}$. We have to integrate this over $(t,z)\in R_{i,j}$ w.r.t.\ the density $\hat h_n$, and then, in the $M$-step, we have to maximize the resulting expression w.r.t.\ $f_{i,j}$. This leads to the following combined $E$-step and (approximate) $M$-step (corresponding to the so-called ``self-consistency equations")
\begin{align*}
f_{k,j}^{(m+1)}=&\sum_{i>k}\int_{t\in A_{n,i},\,z\in B_{n,j}}\frac{\dn f_{k,j}^{(m)}}{\sum_{l:a_{n,l}<t}f_{l,j}^{(m)}\dn+(t-a_{n,i-1})f_{i,j}^{(m)}}\,\hij\,dt\,dv\\
&+\int_{t\in A_{n,k},\,z\in B_{n,j}}\frac{(t-a_{n,k})f_{k,j}^{(m)}}{\sum_{l:a_{n,l}<t}f_{l,j}^{(m)}\dn+\bigl(t-a_{n,k-1}\bigr)f_{k,j}^{(m)}}\hat h_{k,j}\,dt\,dv\\
&+\sum_{i<k}\int_{t\in A_{n,i},\,z\in B_{n,j}}\frac{\dn f_{k,j}^{(m)}}{\sum_{l:a_{n,l-1}>t}F_l^{(m)}\dn+(a_{n,i}-t)F_{i}^{(m)}}\hi\,dt\\
&+\int_{t\in A_{n,k}}\frac{(a_{n,k}-t)f_{k,j}^{(m)}}{\sum_{l:a_{n,l-1}>t}F_l^{(m)}\dn+(a_{n,k}-t)F_{k}^{(m)}}\,\hat h_k\,dt,
\end{align*}
where $\hij$ is the value of $\hat h_n(t,z)$ if $(t,z)\in R_{i,j}$ and $\hi$ is the value of $\hat h_n(t,0)$ if $t\in A_{n,i}$. Hence
\begin{align*}
f_{k,j}^{(m+1)}=&f_{k,j}^{(m)}\sum_{i>k}\log\left(1+f_{i,j}^{(m)}\bigm/\sum_{l<i}f_{l,j}^{(m)}\right)\,\frac{\hij}{f_{i,j}^{(m)}}\dn\en\\
&+f_{k,j}^{(m)}\left\{f_{k,j}^{(m)}-\sum_{l<k}f_{l,j}^{(m)}\log\left(1+f_{k,j}^{(m)}\bigm/\sum_{l<k}f_{l,j}^{(m)}\right)\right\}\,\frac{\hat h_{k,j}}{\left(f_{k,j}^{(m)}\right)^2}\dn\en\\
&+f_{k,j}^{(m)}\sum_{i<k}\log\left(1+F_i^{(m)}\bigm/\sum_{l>i}F_l^{(m)}\right)\,\frac{\hi}{F_i^{(m)}}\dn\\
&+f_{k,j}^{(m)}\left\{F_k^{(m)}-\sum_{l>k}F_l^{(m)}\log\left(1+F_k^{(m)}\bigm/\sum_{l>k}F_l^{(m)}\right)\right\}\,\frac{\hat h_k}{\left(F_k^{(m)}\right)^2}\dn,
\end{align*}
for $1<k<k_n$. For $k=1$ we get
\begin{align*}
f_{1j}^{(m+1)}=&\hat h_{1,j}\dn\en+f_{1,j}^{(m)}\sum_{i>1}\log\left(1+f_{i,j}^{(m)}\bigm/\sum_{l<i}f_{l,j}^{(m)}\right)\,\frac{\hij}{f_{i,j}^{(m)}}\dn\en\\
&+f_{1,j}^{(m)}\left\{F_1^{(m)}-\sum_{l>1}F_l^{(m)}\log\left(1+F_1^{(m)}\bigm/\sum_{l>1}F_l^{(m)}\right)\right\}\,\frac{\hat h_1}{\left(F_1^{(m)}\right)^2}\dn,
\end{align*}
and for $k=k_n$
\begin{align*}
f_{k_n,j}^{(m+1)}=&f_{k_n,j}^{(m)}\left\{f_{k_n,j}^{(m)}-\sum_{l<k_n}f_{l,j}^{(m)}\log\left(1+f_{k_n,j}^{(m)}\bigm/\sum_{l<k_n}f_{l,j}^{(m)}\right)\right\}\,\frac{\hat h_{k_n,j}}{\left(f_{k_n,j}^{(m)}\right)^2}\dn\en\\
&+f_{k_n,j}^{(m)}\sum_{i<k_n}\log\left(1+F_i^{(m)}\bigm/\sum_{l>i}F_l^{(m)}\right)\,\frac{\hi}{F_i^{(m)}}\dn+f_{k_n,j}^{(m)}\,\frac{\hat h_{k_n}}{F_{k_n}^{(m)}}\dn.
\end{align*}

These iterations were used until the absolute value of the scalar product of the vector of values $f_{i,j}^{(m)}$ with the vector of values of partial derivatives of the criterion function w.r.t.\ $f_{i,j}^{(m)}$ was smaller than $10^{-10}$ (here we use the so-called Fenchel duality condition). The algorithm is very fast and can easily be used for simulation purposes, also with sample sizes like $n=10\,000$.

Note that
\be\log\left(1+f_{i,j}^{(m)}\bigm/\sum_{l<i}f_{l,j}^{(m)}\right)\,\frac{\hij}{f_{i,j}^{(m)}}
\sim \frac{\hij}{\sum_{l<i}f_{l,j}^{(m)}},\qquad f_{i,j}^{(m)}\bigm/\sum_{l<i}f_{l,j}^{(m)}\downarrow0,\ee
and
\be\left\{f_{k,j}^{(m)}-\sum_{l<k}f_{l,j}^{(m)}\log\left(1+f_{k,j}^{(m)}\bigm/\sum_{l<k}f_{l,j}^{(m)}\right)\right\}\,\frac{\hat h_{k,j}}{\left(f_{k,j}^{(m)}\right)^2}\sim \frac{\hat h_{k,j}}{2\sum_{l<k}f_{l,j}^{(m)}},\qquad f_{k,j}^{(m)}\bigm/\sum_{l<k}f_{l,j}^{(m)}\downarrow0,\ee
which allows individual $f_{k,j}^{(m)}$ to tend to zero during the iterations. Similar relations hold for the $F_i^{(m)}$.

\bibliographystyle{mystyle}
\bibliography{references}

\end{document}

%% file: commands.tex
\def\A{\ensuremath{\mathcal A}}
\def\B{\ensuremath{\mathcal B}}
\def\C{\ensuremath{\mathcal C}}
\def\E{\ensuremath{{\mbox{\rm E}}}}
\def\F{\ensuremath{\mathcal F}}
\def\H{\ensuremath{\mathcal H}}
\def\K{\ensuremath{\mathcal K}}
\def\N{\ensuremath{\mathcal N}}
\def\NN{\ensuremath{\mathbbm{N}}}
\def\cP{\ensuremath{\mathcal P}}
\def\R{\ensuremath{\mathbbm{R}}}
\def\W{\ensuremath{\mathcal W}}

\def\conv{\ensuremath{\mathop{\rm \longrightarrow}}}

\def\argmax{\mathop{\arg\!\max}\limits}

\newcommand{\eps}{\ensuremath{{\varepsilon }}}

\renewcommand{\phi}{\ensuremath{{\varphi}}}

\newcommand{\Gn}[1]{\ensuremath{\mathbbm{G}_{#1}}}
\newcommand{\Hn}[1]{\ensuremath{\mathbbm{H}_{#1}}}

\newcommand{\FMS}{\ensuremath{\hat F_n^{MS}}}
\newcommand{\fMS}{\ensuremath{\hat f_n^{MS}}}

\newcommand{\hMS}{\ensuremath{\hat h_n^{MS}}}
\newcommand{\FMSX}{\ensuremath{\hat F_{n,X}^{MS}}}
\newcommand{\Fpi}[1]{\ensuremath{\hat F_n^{(#1)}}}

\newcommand{\hf}{\ensuremath{h_{f_0}}}

\newcommand{\hn}{\ensuremath{\hat h_n}}
\newcommand{\fn}{\ensuremath{\bar f_{n,0}}}

\newcommand{\hfn}{\ensuremath{h_{\fn}}}

\newcommand{\hi}{\ensuremath{\hat h_i}}
\newcommand{\hij}{\ensuremath{\hat h_{i,j}}}
\newcommand{\ai}{\ensuremath{\alpha_i}}
\newcommand{\bij}{\ensuremath{\beta_{i,j}}}
\newcommand{\dn}{\ensuremath{\delta_n}}
\newcommand{\en}{\ensuremath{\eps_n}}

\newcommand{\be}{\begin{eqnarray*}}
\newcommand{\ee}{\end{eqnarray*}}
\newcommand{\bef}{\begin{eqnarray}}
\newcommand{\eef}{\end{eqnarray}}

\newcommand{\norm}[2]{\mbox{$\|#1\|_{#2}$}}
\newcommand{\norminfty}[1]{\mbox{$\norm{#1}{\infty}$}}

\newcommand{\Dconv}{\ensuremath{\leadsto}}
\newcommand{\Pconv}{\ensuremath{\stackrel{\cP}{\conv}}}

%% file: theorems.tex
\newtheorem{theorem}{Theorem}[section]

\newtheorem{lemma}[theorem]{Lemma}

\newtheorem{remark}[theorem]{Remark}

\newenvironment{proof}{\noindent\textbf{Proof:}}{$\hfill\square$\vspace{3mm}}
\newenvironment{proof2}[1]{\noindent\textbf{Proof of #1:}}{$\hfill\square$\vspace{3mm}}